\documentclass[12pt]{amsart}

\usepackage{amssymb}
\usepackage{bm}
\usepackage{graphicx}
\usepackage{psfrag}
\usepackage[usenames,dvipsnames]{xcolor}
\usepackage{enumerate}
\usepackage{multirow}
\usepackage{url}

\usepackage{subcaption}
\usepackage{comment}

\usepackage{algpseudocode}

\usepackage{mathtools}

\usepackage{xy}
\input xy
\xyoption{all}

\usepackage{tikz}

\newcommand{\excise}[1]{}

\newtheorem{thm}{Theorem}[section]
\newtheorem{lemma}[thm]{Lemma}

\newtheorem{cor}[thm]{Corollary}
\newtheorem{prop}[thm]{Proposition}

\newtheorem{prob}[thm]{Open Problem}

\newtheorem{mainthm}[thm]{Main Theorem}
\newtheorem*{mainquestion}{Main Question}

\theoremstyle{definition}

\newtheorem{example}[thm]{Example}
\newtheorem{remark}[thm]{Remark}
\newtheorem{defn}[thm]{Definition}

\numberwithin{equation}{section}

\renewcommand\>{\rangle}

\newcommand\CC{\mathbb{C}}

\newcommand\RR{\mathbb{R}}

\newcommand\ZZ{\mathbb{Z}}

\newcommand\kk{\Bbbk}

\DeclareMathOperator\convhull{conv} 

\DeclareMathOperator\lcm{lcm} 
\DeclareMathOperator\supp{supp} 

\DeclareMathOperator\sg{Sg}
\DeclareMathOperator\cone{Cone}
\DeclareMathOperator\spann{span}

\begin{document}

\mbox{}
\title[Convexity in (colored) affine semigroups]{Convexity in (colored) affine semigroups}

\author[De Loera]{Jes\'us A. De Loera}
\address{Mathematics Department\\University of California Davis\\Davis, CA 95616}
\email{deloera@math.ucdavis.edu}

\author[O'Neill]{Christopher O'Neill}
\address{Mathematics Department\\San Diego State University\\San Diego, CA 92182}
\email{cdoneill@sdsu.edu}

\author[Wang]{Chengyang Wang}
\address{Mathematics Department\\University of California Davis\\Davis, CA 95616}
\email{cyywang@ucdavis.edu}

\makeatletter
  \@namedef{subjclassname@2020}{\textup{2020} Mathematics Subject Classification}
\makeatother


\date{\today}

\dedicatory{Dedicated to Bernd Sturmfels on the occasion of his 60th birthday.}

\begin{abstract}
In this paper, we explore affine semigroup versions of the convex geometry theorems of Helly, Tverberg, and Carath\'eodory.  
Additionally, we develop a new theory of \emph{colored affine semigroups}, where the semigroup generators each receive a color and the elements of the semigroup take into account the colors used (the classical theory of affine semigroups coincides with the case in which all generators have the same color).  We prove an analog of Tverberg's theorem and colorful Helly's theorem for semigroups, as well as a version of colorful Carath\'eodory's theorem for cones.  We also demonstrate that colored numerical semigroups are particularly rich by introducing a colored version of the Frobenius number. 
\end{abstract}

\maketitle


\section{Introduction}
\label{sec:intro}

A \emph{semigroup} is a set with an associative binary operation. For an integer matrix $A \in \ZZ^{d\times n}$, we consider the additive semigroup
$$\sg(A) = \{b \in \RR^d : Ax=b, \text{for some} \, \, x \in \ZZ^n, x \geq 0 \}$$
of all non-negative integer combinations of the column vectors of $A$, known as the \emph{affine semigroup} generated by $A$ (see \cite{BGT} for an introduction).  
We also write $S = \sg(V)$ for a set of vectors $V = \{v_1, \ldots, v_n\}$ to denote the semigroup generated by the matrix with column vectors $v_1, \ldots, v_n$.  
Affine semigroups lie in the intersection of algebraic geometry, combinatorics, commutative algebra, convex discrete geometry and number theory.  They are the combinatorial building blocks of toric varieties~\cite{coxlittlehenk}, and they find countless applications in optimization and number theory~\cite{barvinokconvexity,Barvinokbookzurich, BeckRobins:extension, bertsimas-weismantel:ip-book, BGT, cassels1997,eisenbrand:50-years-ip-geom-num-chapter, stanley0, sturmfels}.  At its core, the affine semigroup $\sg(A)$ is the algebraic-combinatorial analogue of its \emph{associated (convex polyhedral) cone}
$$\cone(A) = \{b: Ax = b \text{ for some } x \in \RR^n, x \geq 0 \},$$
consisting of all non-negative real combinations of the columns of $A$.  
In particular, the classical linear Diophantine problem 
\begin{equation}\label{eq:feas1}
A x = b, \, x \geq 0, \, x \in \ZZ^n
\end{equation}
has a solution if and only if $b \in \sg(A)$.  

This paper explores the following question (the theorems mentioned therein are discussed in detail in the paragraphs that follow).  

\begin{mainquestion}
How far can one generalize the convex geometry theorems of Helly, Tverberg, and Carath\'eodory to affine semigroups?
\end{mainquestion}


Helly's theorem, a basic result in convex geometry, states that given a finite family~$\mathcal F$ of convex sets in $\mathbb{R}^{d}$, if every collection of $d+1$ sets in $\mathcal F$ intersect, then the whole family intersects~\cite{DimHelly}. Helly-type theorems appear in many variations~\cite{barany2022helly,DimHelly,JesusSurvey}; for example, Doignon's theorem, an integer version of Helly's theorem, states that if every collection of $2^{d}$ sets in $\mathcal F$ intersect at an integer point, then the whole family intersects at an integer point~\cite{Doignon}.  Our first result is another such variation, one for affine semigroups.  

Recall that a subsemigroup $S = \sg(A) \subseteq \ZZ^m$ is \emph{pointed} if it has no nontrivial subgroups.  This is equivalent to require that $\cone(A)$ contains no positive dimensional linear subspace of $\RR^d$.  We say $S$ is \emph{trivial} if $S = \{0\}$.  

\begin{mainthm}[A Helly theorem for affine semigroups]\label{mt:helly}
For each $m \in \ZZ_{\ge 0}$, there exists a constant $N(m) \in \ZZ_{\ge 1}$ such that the following holds:\ given any finite family $\mathcal{F} = \{ S_{1} ,\ldots,S_{n}\}$ of affine semigroups in $
\ZZ^{m}$, and letting $C_i = \cone(S_i)$ for each $i$, if the intersection of any $N(m)$ affine semigroups in $\mathcal{F}$ is nontrivial, then $S_1 \cap \cdots \cap S_n$ is nontrivial.  More specifically,

\begin{enumerate}[(a)]
\item if each $S_i$ is pointed and $C_1, \ldots, C_n$ do not cover $\RR^{m}$, then $N(m) = m$;
\item if each $S_i$ is pointed and $C_1, \ldots, C_n$ cover $\RR^{m}$, then $N(m) = m + 1$; and
\item if some $S_i$ is not pointed, then $N(m) = 2m$.
\end{enumerate}
\end{mainthm}

Carath\'eodory's theorem says that given a pointed cone $C \subseteq \mathbb{R}^d$, every element $x \in C$ is generated by at most $d$ extreme rays of $C$. (see~\cite{DimHelly} for various variations)  In wide contrast to Main Theorem~\ref{mt:helly}, obtaining a variant of Carath\'eodory's theorem for affine semigroups (i.e., a bound on the number of generators needed to generate any given element) is much more complicated.  In particular, for general affine semigroups, it is impossible to obtain such a bound in terms of ambient dimension $d$ alone; one must also take into account, for instance, the coordinates of the semigroup generators (see \cite{sparse} and all the references there).  A special case of particular interest is when $S = \sg(A)$ is {\em normal}, i.e., if $\sg(A) = \cone(A) \cap \Lambda$ for some lattice $\Lambda$; in this case, every element of $S$ can be generated by at most $2d-2$ generators, though this bound is not tight~\cite{Sebo90}.  


Colorful variations of Helly's, Carath\'eodory's, and Tverberg's theorems have been a key topic in combinatorial convexity~\cite{BaranyConvexity,JesusSurvey}.  
In this vein, we introduce \emph{colored affine semigroups}, wherein each column of the generating matrix $A$ receives one of $\ell$ different colors.  Such a coloring naturally gives a partition $\{\mathcal{I}_{i}\}_{i=1}^{\ell}$ of the $n$ columns of $A$, and we denote the colored affine semigroup as $\sg(A_{1},A_{2},\ldots,A_{\ell})$, where $A_i$ denotes the submatrix of $A$ with color $i$.  In~\cite{ColorfulCaratheodory,BaranyOnn97,Sarrabezolles}, the same type of ideas are studied for \emph{real solutions}; here, we require \emph{integer solutions}, and the theory becomes more subtle.  

Parts~(a) and~(c) of the following definition also appear in~\cite{BaranyOnn97,Rado}.  

\begin{defn}\label{d:colorfulelements}
Fix a solution vector $x$ of $Ax = b$.  The \emph{support} of $x$, denoted $\supp(x)$, is the set of indices $i$ such that $x_{i} \neq 0$, and we say $x$ \emph{uses} a color $c$ if $i \in \supp(x)$ for some $i \in \mathcal{I}_c$.  We say $x$ is:

\begin{enumerate}[(a)]
\item $k$-\emph{chromatic} if $x$ uses at least $k$ different colors;
\item \emph{monochromatic} if $x$ is not $2$-chromatic (i.e., $\supp(x) \subseteq \mathcal{I}_{i}$ for some $i$); 
\item \emph{chromatic} if $x$ uses all available colors (i.e., $|\supp(x) \bigcap \mathcal{I}_{i}| \geq 1$ for all $i$); and
\item \emph{colorful} if no 2 columns of identical color are used (i.e., $|\supp(x) \bigcap \mathcal{I}_{i}| \leq 1$ for all~$i$).
\end{enumerate}
\end{defn}

Some of the terms in Definition~\ref{d:colorfulelements} have subtle distinctions. Example~\ref{e:colorfulelements} shows some differences, and in particular that neither chromatic nor colorful implies the other. 

\begin{example}\label{e:colorfulelements}
Let $A = [9 \,\, 16 \,\, 11 \,\, 14 \,\, 12 \,\, 13]$ and $b = 70$.  Let $\mathcal{I}_{1} = \{1,2\}$, $\mathcal{I}_{2} = \{3,4\}$, $\mathcal{I}_{3} = \{5,6\}$ be a 3-coloring of $A$. Consider
$$ Ax = b, \quad x \geq 0, \quad x \in \mathbb{Z}^{6}.$$
The solution $x = (6,1,0,0,0,0)$ is monochromatic.  
The solution $x = (3,1,0,1,0,1)$ is chromatic since each color is used, but not colorful since two distinct columns from $\mathcal I_1$ are used.  
The solution $x = (0,1,0,2,0,2)$ is both chromatic and colorful, since exactly one column is used from each color.  
Lastly, the solution $x = (0,0,2,0,4,0)$ is colorful, 2-chromatic, but not chromatic.
\end{example}

We briefly argue that the above notions arise naturally when modeling manufacturing diversity requirements.  
The notion of colorful has already been connected to linear programming and game theory in \cite{JesusSurvey,Sarrabezolles}.  When dealing with indivisible goods, this kind of integer programming requires affine semigroups.  
Imagine your company produces batteries with three ingredient providers (call them red, green, and blue).  They each sell exactly the same resources or ingredients to you, which are represented by vectors (say different types of metals or chemicals). But due to trade agreements, one cannot produce a battery with parts coming from one provider alone (no monochromatic solutions are allowed).  Since batteries must be built with parts from at least two providers, solutions then have to be 2-chromatic.  Or regulations can be even more strict, requiring batteries to be built with ingredients from all three providers (chromatic solutions).  Another possible type of restriction is that a company may only contribute at most one ingredient to the creation of your product (colorful solutions).  In some scenarios, it should be possible to purchase the same ingredient from different providers to cover demand.  As such, we allow the same column to appear more than once, but with a different color.

Colorful versions of Helly's and Tverberg's theorems for affine semigroups follow from Main 
Theorem~\ref{mt:helly} (Corollaries~\ref{c:colorhelly} and~\ref{c:tverberg}),
but obtaining a colorful version of Carath\'eodory's theorem for affine semigroups turns 
out to be a bit more subtle.  

With the above definitions in hand, we recall a colorful variation of Carath\'eodory's theorem due to B\'ar\'any.  
Given $d+1$ nonempty subsets $I_1, \ldots, I_{d+1} \subseteq \RR^d$, B\'ar\'any's theorem states that any point $x \in \convhull(I_1) \cap \cdots \cap \convhull(I_{d+1})$ can be expressed as the convex combination of $d+1$ points, with one point from each $I_j$~\cite{ColorfulCaratheodory}.
Considering each set $I_j$ as a color class, B\'ar\'any's theorem has the following interpretation:\ given a colored generating matrix $A$ and an element $b \in \cone(A)$, if a monochromatic solution exists for each color, then a colorful solution exists.  In discrete convexity, the colorful Carath\'eodory theorem has been intensely studied~\cite{DMS,DSX,MMSS}.  

Returning once again to affine semigroups, suppose an element $b$ of a colored affine semigroup has a monochromatic solution for each color.  Can one guarantee $b$ also has a colorful solution?  What about a chromatic solution?  Note, an answer of ``yes'' to either question would constitute a variant of B\'ar\'any's theorem for affine semigroups.  
It~turns out, the answer to the latter question is indeed ``yes'' for all but finitely many~$b$ (Main Theorem~\ref{mt:colorcaratheodory}), but the former question has a overwhelmingly negative answer, as the following result demonstrates in two different ways.  Note that the families described therein can be easily lifted to higher dimensions.  


\begin{mainthm}\label{mt:failureofcaratheodory}
B\'ar\'any's \textbf{colorful} Carath\'eodory theorem fails to extend to affine semigroups.  
\begin{enumerate}[(a)]
\item 
There exist colored affine semigroups with arbitrarily many colors in $\RR^3$, formed by a family $\mathcal{F}$ of normal affine semigroups and an element $b$ such that $b$ has a monochromatic solution for every color but yet has no colorful solutions and no chromatic solutions (in fact, every solution for $b$ is monochromatic).

\item
There exist colored affine semigroups with arbitrarily many colors in $\RR^4$, formed by a family $\mathcal{F}$ of normal affine semigroups and infinitely many elements $b$ such that $b$ has a monochromatic solution for every color and has no colorful solution.
\end{enumerate}
\end{mainthm}

We now turn our attention to chromatic solutions.  
Main Theorem~\ref{mt:failureofcaratheodory}(a) demonstrates the ``all but finitely many'' hypothesis in Main Theorem~\ref{mt:colorcaratheodory} cannot be dropped.  We note that this hypthesis may seem unnatural to those in convexity theory, but such theorems arise frequently in semigroup theory, where the finitely many exceptions can be attributed to the important notion of \emph{gaps} or \emph{holes} describing exceptions in semigroup membership~\cite{gapsref}.  

\begin{mainthm}[A \textbf{chromatic} Carath\'eodory theorem for affine semigroups]\label{mt:colorcaratheodory}
In~any colored affine semigroup $S$, all but finitely many elements $b \in S$ with a monochromatic solution for each color also have a chromatic solution.  
\end{mainthm}


In Section~\ref{sec:numericalsemigroups}, we consider the special case of \emph{numerical semigroups}~\cite{Rosales}, namely, when the matrix $A=(a_1,\ldots,a_n)\in\ZZ_{>0}^n$ is a positive integral $n$-dimensional vector that is \emph{primitive} (i.e., $\gcd(A) = 1$).  Often referred as a \emph{knapsack problem}~\cite{kellereretal}, numerical semigroups are fundamental and look simple, but are often a source of very challenging problems \cite{aardaletal3}.  
One old and classical problem
is the \emph{Frobenius coin-exchange problem}, which asks for the largest integer $F(A) = b$ that cannot be expressed as a non-negative integral combination of the $a_i$'s \cite{frob} (here, the $a_i$'s' are interpreted as coin values, and we are looking for the largest value $b$ for which one cannot make even change). 

Here, we consider the \emph{\textbf{chromatic} Frobenius problem}:\ 
if we assign one of $\ell$ different colors to each $a_i$, then the challenge is to find the largest $b$ such that $b \in \sg(A_{1},\ldots,A_{\ell})$ but no solution $b=Ax$ uses all distinct colors.  
More specifically, we define the \emph{$k$-chromatic Frobenius number} of $S = \sg(A_{1},\ldots,A_{\ell})$ as the largest integer $b = \mathsf{CF}_{k}(S)$ with no $k$-chromatic solution. 
Our main results in this direction are as follows. 

\begin{mainthm}\label{mt:numerical}
Fix $k \ge 1$ and a colored numerical semigroup $S = \sg(A_{1},\ldots,A_{\ell})$.  
\begin{enumerate}[(a)]
\item 
\label{colorFrobeniusFT}
There are only finitely many elements of $S$ that are not $k$-chromatic, and as such, the $k$-chromatic Frobenius number $\mathsf{CF}_{k}(S)$ is well defined. 

\item 
\label{colorFrobeniusNP}
Computing the colored Frobenius number $\mathsf{CF}_{k}(S)$ is NP-hard. 

\item 
\label{cardcolorsolutions}
The number of distinct $k$-chromatic solutions of a positive integer $b$ coincides with a quasipolynomial function in $b$ for sufficiently large $b$.  

\end{enumerate}
\end{mainthm}

\section{Helly and Tverberg theorems for semigroups}
\label{sec:withcolors}

In order to prove Main Theorem~\ref{mt:helly}, we recall three fundamental results about affine semigroups that, together, ensure a nontrivial intersection of affine semigroups occurs precisely when their associated cones intersect nontrivially (Proposition~\ref{p:capequivalent}).  

\begin{lemma}[{\cite[Corollary~2.11(a)]{BrunsPolytopes}}]\label{l:fgcapfg}
The intersection of two affine semigroups is again an affine semigroup.
\end{lemma}

\begin{remark}\label{r:computeintersection}
If a semigroup $S = \cone(A) \cap \Lambda$ for some lattice $\Lambda$, then the minimal generating set of $S$, called the \emph{Hilbert basis} of $A$, can be computed~\cite{brunskoch}.  
If~the generators of two affine semigroups $S_1$ and $S_2$ are given via matrices $A$ and $B$, then one can compute the generators of $S_1 \cap S_2$ by constructing a rational cone 
$$C = \{(x,y) \ge 0 : Ax - By = 0\},$$
finding its Hilbert basis, and then mapping each Hilbert basis element $(x,y) \mapsto Ax$. 
\end{remark}

\begin{lemma}\label{l:ktimesp}
For any affine semigroup $S \subset \ZZ^d$, if $\cone(S)$ contains an integral point~$p$, then $k p \in S$ for some positive integer $k$.
\end{lemma}

\begin{lemma}\label{l:cpcommute}
For affine semigroups, taking finite intersections commutes with taking the associated cone:\ if $S_1, \ldots, S_n \subset \ZZ^d$ are affine, then $\cone(\bigcap_{i=1}^{n} S_{i}) = \bigcap_{i=1}^{n} \cone(S_{i})$.
\end{lemma}


\begin{prop}\label{p:capequivalent}
The intersection $\bigcap_i S_i$ of affine semigroups $S_1, \ldots, S_n \subset \ZZ^d$ contains a non-zero element if and only if $\bigcap_i \cone(S_i)$ contains a non-zero element.  
\end{prop}

\begin{proof}
Apply all parts of Lemma Lemmas~\ref{l:fgcapfg},~\ref{l:ktimesp} and~\ref{l:cpcommute}.
%
%
\end{proof}

\begin{proof}[Proof of Main Theorem \ref{mt:helly}]

Let $C_i = \cone(S_i)$ for each $i$, and let $\mathcal G = \{C_1, \ldots, C_n\}$. 
By Proposition~\ref{p:capequivalent}, it suffices to show in each case that $C_1 \cap \cdots \cap C_n$ is nontrivial. 

\begin{enumerate}[(a)]
\item 
Consider the unit sphere $\mathbb S^{d-1} \subset \RR^d$.  If a convex cone $C$ contains a non-zero element, then $C$ will intersect the unit sphere.  Hence, instead of proving the family $\mathcal{G}$ intersects at a non-zero element, it suffices to prove $\mathcal G' =\{C_{1} \cap\mathbb{S}^{d-1} ,\ldots,C_{n}\cap \mathbb{S}^{d-1}\}$ has nonempty intersection.  Since $\mathcal G$ does not cover $\RR^d$, $\mathcal G'$ does not cover $\mathbb{S}^{d-1}$.  Therefore, for any point $q \in \mathbb{S}^{d-1}$ not covered by $\mathcal G'$, there exists a homeomorphism $f :\mathbb{S}^{d-1} \setminus \{q\} \to \mathbb{R}^{d-1}$. under which it suffices to prove the family 
$$\mathcal{G}'' = \{f(C_{1} \cap\mathbb{S}^{d-1}) ,\ldots,f(C_{n}\cap \mathbb{S}^{d-1})\}$$ has nonempty intersection.  To this end, we employ a topological variant of Helly's theorem~\cite{TopHelly}, which states that for a finite family of closed sets in $\RR^d$, if the intersection of every $d+1$ members is contractible, then the intersection of the family is contractible.  

Now, each $C_i$ is a rational polyhedral cone and therefore topologically closed, and if $S_i$ has only the trivial subgroup, then $C_i$ is pointed.  Since each $C_i$ is closed and pointed, so is any intersection of the $C_i$'s.  We can conclude that each set $f(C_{i} \cap \mathbb{S}^{d-1})$ is closed in $\mathbb R^{d-1}$, and in particular that the intersection of any $N = d$ of the sets in $\mathcal G''$ is nonempty and contractible.  As such, applying the aforementioned topological Helly's theorem to $\mathcal G''$ completes the proof. 

\item 
If each $S_i$ has only the trivial subgroup, then each $C_i$ is pointed, and thus $C_i \setminus \{0\}$ is convex for each $i$.  As such, if every $N = d+1$ of the $C_i$'s intersects nontrivially, then the claim in this case follows from Helly's theorem for convex sets in $\RR^d$.

\item 
In this case, we employ a $j$-dimensional variant of Helly's theorem~\cite{DimHelly}, a special case of which states that for a family of finite convex sets in $\RR^{d}$, if the intersection of every $2d$ members is at least 1-dimensional, then the intersection of the family is at least 1-dimensional.  This can be applied directly, as any intersection of rational cones that contains a nonzero point must be at least 1-dimensional.  
\end{enumerate}

In each of the above cases, $C_1 \cap \dots \cap C_n$ contains a non-zero element.
%
\end{proof}

We now illustrate that each choice of $N$ in Main Theorem~\ref{mt:helly} is best possible.

\begin{example}\label{e:mthelly}
Let $e_{i}$ be $i$-th standard basis in $\mathbb{R}^{d}$. 
\begin{enumerate}[(a)]
\item 
Let $E = \{e_1, \ldots, e_d\}$, and consider the affine semigroups $S_i = \sg(E \setminus \{e_i\})$.  
The intersection of any $d-1$ contains a non-zero element, as 
$e_{i} \in \bigcap_{j\neq i} S_j$
for each $i$, but the intersection of all $d$ affine semigroups is trivial.

\item 
Let $P$ be any $d$-simplex with the origin in its interior and vertices set denoted $V = \{v_1, \ldots, v_{d+1}\}$, and consider the affine semigroups $S_i = \sg(V\setminus \{v_i\})$.  
We can verify that
$v_i \in \bigcap_{j\neq i} S_j$
for each $i$, but $\bigcap_j S_j$ is trivial.  

\item 
Let $E = \{e_1, -e_1, \ldots, e_d, -e_d\}$.  Consider the affine semigroups
$$
S_{i,+} = \sg(E \setminus \{-e_{i}\})
\qquad \text{and} \qquad
S_{i,-} = \sg(E \setminus \{e_{i}\})
$$
for each $i$.  Any $2d - 1$ of the above affine semigroups share a non-zero element, as 
$$\pm e_i \in S_{1,+} \cap S_{1,-} \cap \dots \cap S_{i,\pm} \cap \dots \cap S_{d,+} \cap S_{d,-}$$
for each $i$, but the only point common to all $2d$ is the origin.
\end{enumerate}
\end{example}

\begin{figure}[t]
\centering
\includegraphics[width=0.4\textwidth]{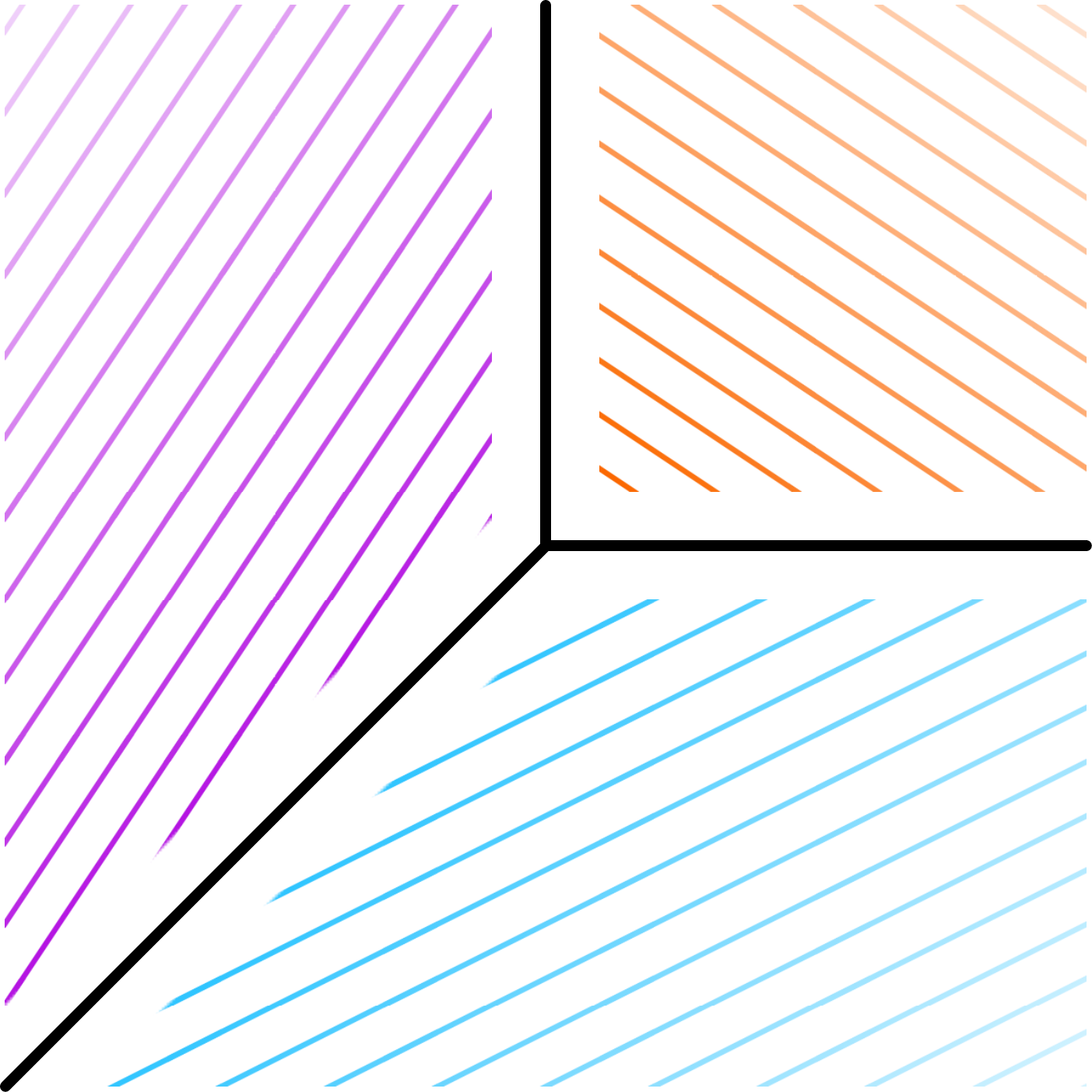}
\hspace{0.5in}
\includegraphics[width=0.4\textwidth]{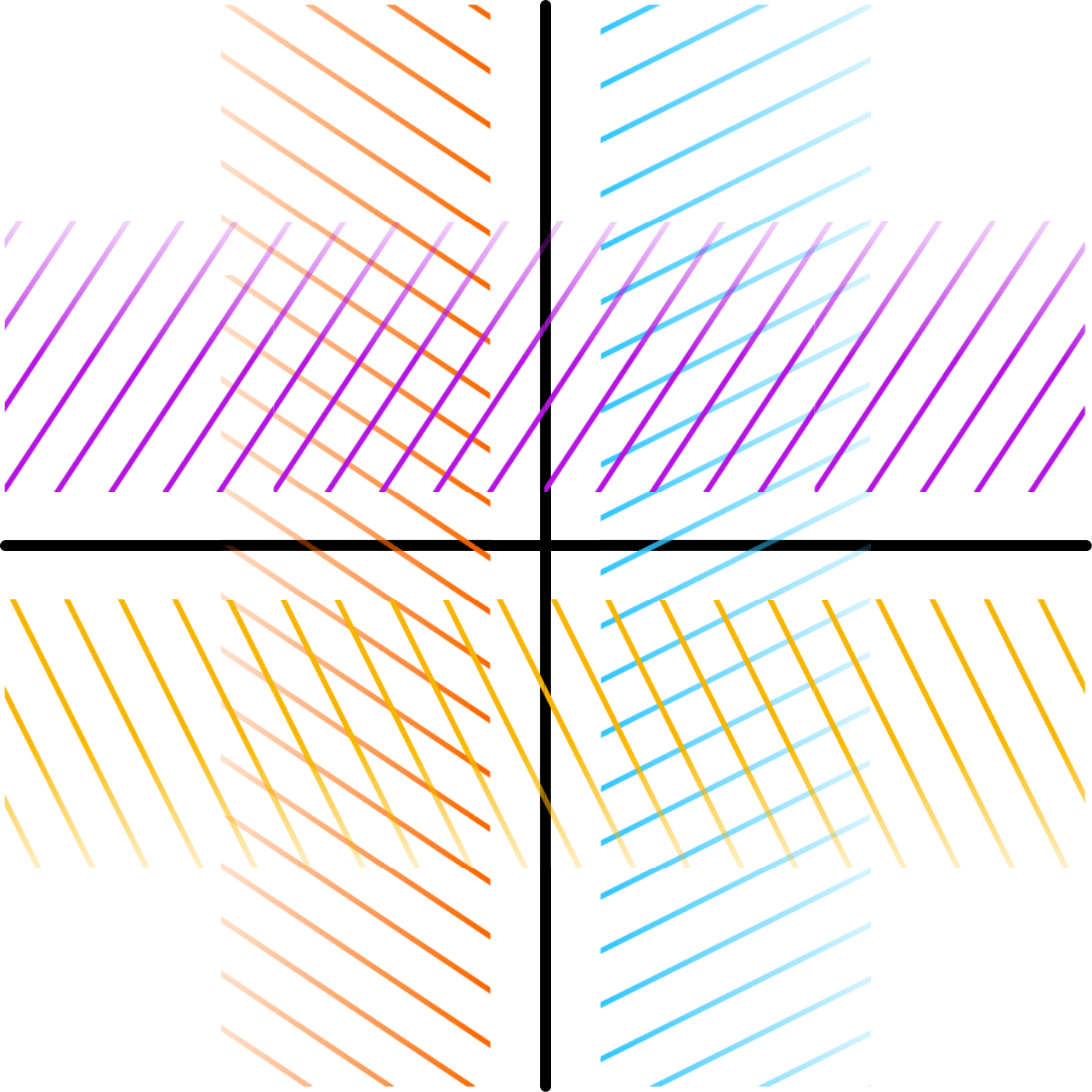}
\caption{Depiction of the families with $d = 2$ in Example~\ref{e:mthelly}(b) (left) and Example~\ref{e:mthelly}(c) (right).}
\label{fig:mthelly}
\end{figure}

We close this section with two corollaries of Proposition~\ref{p:capequivalent}.  The first is an analog of the colorful Helly's theorem~\cite{DimHelly}, which
asserts that given $d+1$ finite families $\mathcal F_1, \ldots, \mathcal F_{d+1}$ of convex sets, if for every choice of transversal $S_{1} \in \mathcal{F}_{1}$, $S_{2} \in \mathcal{F}_{2}$,\ldots, $S_{d+1} \in \mathcal{F}_{d+1}$, the intersection $ S_{1}\cap \dots\cap S_{d+1}$ is nonempty, then for some $j$, the sets in $\mathcal F_j$ have nonempty intersection.  The second is an analog of Tverberg's theorem~\cite{tverberg}, which states that for any set $D$ of $(d+1)(r-1)+1$ points in $\RR^d$, there exists a point $p$ (not necessarily in $D$) and a partition of $D$ into $r$ blocks, such that $p$ belongs to the convex hull of each block.  Note that both are ``partial'' analogs, as all affine semigroups therein are required to be pointed.

\begin{cor}\label{c:colorhelly}
Let $\mathcal{F}_{1},\ldots,\mathcal{F}_{N}$ be finite families of pointed affine semigroups in $\mathbb Z^d$.  If for every choice of a transversal $S_{1} \in \mathcal{F}_{1}$, $S_{2} \in \mathcal{F}_{2}$,\ldots, $S_{d+1} \in \mathcal{F}_{d+1}$, the intersection $ S_{1}\cap \dots\cap S_{d+1}$ contains a non-zero element, then there is a family $\mathcal{F}_{j}$ such that all semigroups in $\mathcal{F}_{j}$ intersect at a non-zero element. 
\end{cor}

\begin{proof}
Following the proof of Main Theorem~\ref{mt:helly}(b), 
replacing each affine semigroup with its associated cone with origin removed yields $d+1$ families of convex sets in $\RR^d$, to which one can readily apply colorful Helly's theorem~\cite{JesusSurvey}.  
\end{proof}


\begin{cor}\label{c:tverberg}
Fix a pointed affine semigroup $S = \sg(A) \subset \mathbb{Z}^{d}$ given by $|A| = k$ generators.  If $k \ge d(r-1)+1$, then there exists a $r$-coloring of $S$ such that some element $p \in S$ has a monochromatic solution of every color.  
\end{cor}




\begin{proof}

By Proposition \ref{p:capequivalent}, we must show that there exists a partition $A_1, \ldots, A_r$ of $A$ such that some non-zero element $p \in \cone(A)$ lies in $\cone(A_i)$ for each $i$.  
Since $S$ has only the trivial subgroup, $\cone(S)$ is pointed, so by taking a cross-section of $\cone{(A)}$, it is equivalent to show that given a set $D$ of $k$ points in $\mathbb R^{d-1}$, there exists a $r$-coloring of $D$ and a point $p$ that lies in the convex hull of each color class. 
Since $k \ge ((d-1)+1)(r-1)+1$, this is exactly the statement of Tverberg's theorem.
\end{proof}

\begin{prob}\label{prob:hellynonpointed}
Generalize Corollaries~\ref{c:colorhelly} and~\ref{c:tverberg} to families of (not necessarily pointed) affine semigroups.
\end{prob}

\section{Carath\'eodory type theorems for semigroups}
\label{sec:caratheodorysemigroups}

The semigroup version of the colorful Carath\'eodory theorem fails strongly.  We provide two counterexamples; Table~\ref{tb:cteg} contains an example of one, and Example~\ref{e:cteg1} illustrates another.

\begin{table}[t!]
\centering
\begin{tabular}{ c | c @{\ } c @{\ } c }
    & $g_i$ & $g_i'$ & $g_i''$  \\
    \hline
$S_{1}$   &  (0,1,2)     &   (1,7,9)      &    (2,9,9)  \\
$S_{2}$   &  (0,3,4)     &   (1,9,11)     &    (2,5,5)  \\
$S_{3}$   &  (0,7,8)     &   (1,13,15)     &    (2,-3,-3)  \\
$S_{4}$   &  (0,15,16)   &   (1,21,23)    &    (2,-19,-19)  \\
$S_{5}$   &  (0,31,32)   &   (1,37,39)    &    (2,-51,-51)  \\
$S_{6}$   &  (0,63,64)   &   (1,69,71)    &    (2,-115,-115)
\end{tabular}
\caption{The family of semigroups in Proposition~\ref{p:cteg1} with $n = 6$.}
\label{tb:cteg}
\end{table}


\begin{prop}\label{p:cteg1}
Fix $n \ge 1$, and consider the family of semigroups $\mathcal{F}_n = \{ S_i = \sg(g_i, g_i', g_i'') \}$, where
$$
g_i = (0, 2^i - 1, 2^i),
\quad
g_i' = (1, n + 2^i - 1, n + 2^i + 1),
\quad
g_i'' = (2, 2(n - 2^i) + 1, 2(n - 2^i) + 1)
$$
and $1 \le i \le n$.  Letting $p = (3, 3n - 1, 3n + 2)$, we have $p \in S_i$ for each $i$, and 
the only expressions for $p$ as a sum of generators from across the $S_i$'s are those of the form 
$$p = g_i + g_i' + g_i''$$
for each $i$.  
\end{prop}

\begin{proof}
Consider an arbitrary expression for $p$ as a sum of generators from the $S_i$'s.  We~claim any expression for $p$ must have the form $p = g_i + g_j' + g_k''$, where $i$, $j$, and $k$ are not necessarily distinct.  Indeed, some generator $g_j'$ must appear, since the first coordinate of $p$ is odd, and from there, some generator $g_i$ must appear since the last 2 coordinates of $p$ differ by 3.  The~first coordinate of $p$ then forces the third and final generator in the expression to have the form $g_k''$.  This proves the claim.  Examining second coordinates in any such expression, we see $2^i + 2^j = 2^{k+1}$, which is impossible unless $i = j = k$.  
\end{proof}

\begin{example}\label{e:cteg1}
Let $\mathcal{I}_{1} = \{1,2,3\}$, $\mathcal{I}_{2} = \{4,5,6\}$, $\mathcal{I}_{3} = \{7,8,9\}$, and
$$A = \left[\begin{array}{c c c|c c c|c c c} 0 & 1 & 2 & 0 & 1 & 2 & 0 & 1 & 2\\ 0 & 32 & 63 & 1 & 33 & 61 & 3 & 35 & 57\\ 1 & 34 & 63 & 2 & 35 & 61 & 4 & 37 & 57 \end{array}\right] \!\!.$$
The element $p = (3, 95, 98)$ has monochromatic solutions 
$$
(1,1,1,0,0,0,0,0,0), 
\quad
(0,0,0,1,1,1,0,0,0),
\quad \text{ and } \quad
(0,0,0,0,0,0,1,1,1),
$$
but no $3$-chromatic solutions.
\end{example}

\begin{proof}[Proof of Main Theorem \ref{mt:failureofcaratheodory}]
Proposition~\ref{p:cteg1} implies part~(a) upon noting that all affine semigroups therein are normal since their generating matrices have determinant $-1$ (see \cite[Chapter~8, Corollary~2.6]{barvinokconvexity}).  For part~(b), for each family $\mathcal{F}_n$ in Proposition~\ref{p:cteg1}, consider the family 
$$\mathcal F' = \{S \times \ZZ_{\ge 0} : S \in \mathcal{F}_n \},$$
of semigroups of the form
$$
S \times \ZZ_{\ge 0} = \sg((g_1,0), \ldots, (g_r,0),  (0,1))
\qquad \text{whenever} \qquad
S = \sg(g_1, \ldots, g_r).
$$
For each $k \ge 1$, the element $(p,k)$ lies in $S \times \ZZ_{\ge 0}$ for each $S \in \mathcal F$, and the only expressions for $(p,k)$ as a sum of generators of the semigroups in $\mathcal F'$ are obtained by concatenating an expression for $(p,0)$ with $k$ copies of $(0,1)$. According to the proof of Proposition~\ref{p:cteg1}, the only way to generate $(p,0)$ is $(g_{i},0) + (g_{i}',0) + (g_{i}'',0)$ for some index $i$. This solution violates the condition of being colorful because three different vectors of the same color are part of the expression for $(p,k)$.
Note that in the previous construction, $(0,1)$ appeared many times with different colors. The definition of colorful is violated  because the vectors $(g_{i},0),(g_{i}',0), (g_{i}'',0)$ are of the same color. 
\end{proof}

The family in the above proof can be adjusted to use different vectors, of different colors, by replacing the instances of $(0,1)$ with vectors from $(0,1),(0,2),(0,3),\ldots$, so that the infinitely many vectors $(p,k \lcm(1,2,3,\dots,s))$ still do not have a colorful representation since $(p,0)$ does not.


\begin{proof}[Proof of Main Theorem~\ref{mt:colorcaratheodory}]
Let $S = \sg(A_1, \ldots, A_\ell)$ and $S' = \bigcap_{i=1}^\ell \sg(A_i)$.  If an element $b \in S$ has a monochromatic solution of each color, then $b \in S'$, so it suffices to prove there are only finitely many elements of $S'$ with no chromatic solution.
%
%
By Lemma~\ref{l:fgcapfg}, $S'$ is finitely generated, say with minimal generating set $G$.  
Therefore, each $b \in S'$ can be written as $b = \sum_{g \in G} \lambda_g g$ with each $\lambda_g \in \mathbb Z_{\ge 0}$.
We will prove that if $\sum_g \lambda_g \ge \ell$, then $b$ has a chromatic solution.  In fact, under this assumption, we can collect terms in this sum to form an expression 
$b = s_1 + \cdots + s_\ell$
as a sum of $\ell$ nonzero elements of $S'$.  Each $s_i$ thus has a monochromatic solution in $S$ of color $i$, and concatenating these monochromatic solutions yields a chromatic solution for $b$.  
\end{proof}

\section{Colored numerical semigroups}
\label{sec:numericalsemigroups}

In this section, we turn our attention to colored numerical semigroups and the chromatic Frobenius problem.  
Before restricting to this case, however, we prove the following general result, which forms the backbone of the proof of Main Theorem~\ref{mt:numerical} but holds for any colored affine semigroup.

\begin{thm}\label{t:structure}
For a colored affine semigroup $S = \sg(A_{1},\ldots,A_{\ell})$, the set 
$$S(A,k) = \{Ax : x \text{ is $k$-chromatic}\}$$
equals the union of finitely many translated copies of $S$.
\end{thm}

\begin{proof}
Consider the map $\varphi:\ZZ_{\ge 0}^n \to \sg(A)$ sending each standard basis vector $e_i$ to the $i$'th column $Ae_i$ of $A$, and let 
$$E = \varphi^{-1}(S(A,k)) = \{x \in \ZZ_{\ge 0}^n : x \text{ is $k$-chromatic}\}.$$
Note that $E$ is closed under the additive action of $\ZZ_{\ge 0}$, as $x + e_i$ is nonzero in every entry that $x$ is nonzero.  By Dickson's lemma~\cite{Dickson}, any subset of $\ZZ_{\ge 0}^n$ has finitely many minimal elements under the component wise partial order, so 
$$
E = (\ZZ_{\ge 0}^n + x_1) \cup \cdots \cup (\ZZ_{\ge 0}^n + x_r)
$$
for some $x_1, \ldots, x_r \in E$.  Applying $\varphi$ to the above equality completes the proof.  
\end{proof}

\subsection{Chromatic Frobenius numbers}
\label{subsec:coloredfrobenius}

For the remainder of this section, fix a colored numerical semigroup $S = \sg(A_{1},\ldots,A_{\ell})$, where $A_1, \ldots, A_{\ell}$ partition $A = \{a_{1},\ldots,a_{n}\}$ with $\gcd(A) = 1$.  The~set of \emph{gaps} of $S$, denoted $G(A) = \ZZ_{\ge 0} \setminus S$, is then a finite set with $F(A) = \max(G(A))$ (this follows from B\'ezout's identity, see~\cite{Rosales}). 
Analogously, the \emph{$k$-chromatic gaps} are the integers in the set $G(A,k) = \ZZ_{\ge 0} \setminus S(A,k)$, so that 
$\mathsf{CF}_{k}(A_{1},\ldots,A_{\ell}) = \max(G(A,k))$. 

The following provides upper and lower bounds for $\mathsf{CF}_{k}(A_{1},\ldots,A_{\ell})$, and in particular verifies $G(A,k)$ is a finite set, as claimed in Main Theorem~\ref{mt:numerical}\eqref{colorFrobeniusFT}.

\begin{cor}\label{c:welldef}
The colored Frobenius number satisfies
$$\min(m(A,k)) - 1 \leq  \mathsf{CF}_{k}(A_{1},\ldots,A_{\ell}) \leq  \min(m(A,k)) + F(A),$$
where $m(A,k) = \bigcup_{|I| = k} \sum_{i \in I} A_{i}$.  
\end{cor}

\begin{proof}
By Theorem~\ref{t:structure}, $S(A,k)$ equals the union of finitely many translates of copies of $S$.  In fact, more can be said, as
$$S(A,k) = \bigcup_{v \in m(A,k)} S + v.$$
Therefore, if $b \leq \min(m(A,k)) - 1$, then $b \not\in S(A,k)$ since $b \not\in S + v$ for any $v \in m(A,k)$,
and if $b > \min(m(A,k)) + F(A)$, then $b \in S(A,k)$ since $b - \min(m(A,k)) \in S$.
\end{proof}

\begin{proof}[Proof of Main Theorem~\ref{mt:numerical}\eqref{colorFrobeniusFT}]
Apply Corollary~\ref{c:welldef}.  
\end{proof}

We also obtain the following chromatic generalization of the well-known formula $F(a,b) = ab - (a + b)$, which holds whenever $\gcd(a,b) = 1$.  

\begin{cor}\label{c:singletons}
We have
$$ \mathsf{CF}_\ell(\{a_1\}, \ldots, \{a_\ell\}) = a_1 + \cdots + a_\ell + F(A),$$
and in particular $\mathsf{CF}_{2}(\{a_1\},\{a_2\})  = a_1a_2$.
\end{cor}

\begin{proof}
Proceeding as in the proof of Corollary~\ref{c:welldef}, if each $A_{i}$ is a singleton, 
$$S(A,\ell) = S + a_1 + \cdots + a_\ell,$$
and as such, $\mathsf{CF}_{\ell}(\{a_{i}\},\ldots,\{a_{\ell}\}) = a_1 + \cdots + a_\ell + F(A)$.  When $\ell = 2$, this then yields $\mathsf{CF}_{2}(\{a\},\{b\}) = F(A) + a + b =ab.$
\end{proof}

\begin{remark}\label{e:nonsingletons}
The chromatic Frobenius number is not always represented as the Frobenius number and some generators from each color class. For instance, $\mathsf{CF}_{2}(\{a,c\},\{b\}) = ab$ whenever $c > ab$.  
\end{remark}

Before proving Main Theorem~\ref{mt:numerical}\eqref{colorFrobeniusNP}, we prove the following lemma.  

\begin{lemma}\label{l:estimate}
If $1 \leq k < \ell$, then
$\mathsf{CF}_{k}(A) \leq  \mathsf{CF}_{k+1}(A)$.
Moreover, if $\gcd(A \setminus A_i) = 1$, so that $F(A \setminus A_{i})$ and $\mathsf{CF}_{\ell-1}(A_{1},\ldots,A_{i-1},A_{i+1},\ldots,A_{\ell})$ both exist, then 
\begin{align*}
\mathsf{CF}_{\ell}(A_{1},\ldots,A_{\ell})
& \leq \mathsf{CF}_{\ell-1}(A_{1},\ldots,A_{i-1},A_{i+1},\ldots,A_{\ell}) + \min{A_{i}} \\
& \leq \mathsf{CF}_{\ell}(A_{1},\ldots,A_{\ell}) + F(A \setminus A_{i}) + 1.
\end{align*}
\end{lemma}

\begin{proof}
The first claim follows from the fact that $S(A,k) \subseteq S(A,k+1)$.  

%
%

In what follows, let $A = (A_{1},\ldots,A_{\ell})$ and $B = (A_{1},\ldots,A_{i-1},A_{i+1},\ldots,A_{\ell})$.  For the first inequality, we must prove that if $b > \mathsf{CF}_{\ell-1}(B) + \min{A_{i}}$, then $b \in S(A,\ell)$.  Since $b - \min{A_{i}} > \mathsf{CF}_{\ell-1}(B)$, we know $b - \min{A_{i}}\in S(B,\ell-1)$, so we can write
$$b - \min{A_{i}} = a_1' + \dots + a_{i-1}' + a_{i+1}' + \dots + a_{\ell}' + c,$$
where $c \in \sg(A\setminus A_{i})$ and each $a_k' \in A_{k}$.  This implies $b \in S(A,\ell)$.  


For the final inequality, we must prove that if $b >  \mathsf{CF}_{\ell}(A) + F(A \setminus A_{i}) - \min{A_{i}} + 1$, then $b \in S(B,\ell-1)$.  Since
$b - (F(A \setminus A_{i}) - \min{A_{i}} + 1) \in S(A,\ell)$, we can write
$$b - (F(A \setminus A_{i}) - \min{A_{i}} + 1) =  a_1' + \cdots + a_i' + \cdots + a_\ell' + c,$$ 
where $c \in \sg(A)$ and each $a_k' \in A_k$.  Notice $a_i' - \min{A_{i}} + 1 > 0$ and $c \ge 0$, which imply
$$c' = a_i' - \min{A_{i}} + 1 + F(A \setminus A_{i}) + c  > F(A \setminus A_{i})$$
and in particular $c' \in \sg(A \setminus A_{i})$.  Hence, 
$$ b = a_1' + \cdots + a_{i-1}' + a_{i+1}' + \cdots + a_\ell' + c' \in S(B, \ell-1),$$ 
as desired.  
\end{proof}




\begin{proof}[Proof of Main Theorem~\ref{mt:numerical}\eqref{colorFrobeniusNP}]
For each $k \in \ZZ_{>0}$, let $P(k)$ be the statement that computing $\mathsf{CF}_{k}(A_{1},A_{2},\ldots,A_{\ell})$ is NP-hard for all $\ell \geq k$.
We will prove the statement by induction on $k$. 
First, when $k = 1$, the colored Frobenius number coincides with the classical Frobenius number, so the statement $P(1)$ is true since the computation complexity of the classical Frobenius number is NP-hard \cite{frob}.

For the inductive step, supposing the statement $P(m)$ is true, we will find a polynomial-time reduction to prove $P(m+1)$.  We do so by proving that there exists a natural number $b$, which can be found in polynomial time, such that:
\begin{enumerate}[(a)]
\item 
if $\ell > k$, then
$\mathsf{CF}_{k+1}(2A_{1},\ldots,2A_{\ell},\{b\}) = 2\mathsf{CF}_{k}(A_{1}.,\ldots,A_{\ell}) + b$ (here, $2A = A + A$); and

\item 
if $\ell = k$, then
$\mathsf{CF}_{\ell+1}(A_{1},\ldots,A_{\ell},\{b\}) = \mathsf{CF}_{\ell}(A_{1},\ldots,A_{\ell}) + b$.
\end{enumerate}
Indeed, the above claims immediately yield a polynomial-time reduction, so 
the statement $P(m+1)$ is true.

For simplicity, let $A = (A_{1},\ldots,A_{\ell})$.  We will prove claim~(a) by proving the following statement:\ if $\ell > k$, then for any odd $b$ with $b > 2\mathsf{CF}_{k+1}(A) -  2\mathsf{CF}_{k}(A)$ and $b >  2\mathsf{CF}_{k}(A)$, 
$$\mathsf{CF}_{k+1}(2A_{1},\ldots,2A_{\ell},\{b\}) = 2\mathsf{CF}_{k}(A_{1},\ldots,A_{\ell}) + b.$$
If a number $p > 2\mathsf{CF}_{k}(A) + b$, then by the choice of $b$, $p > 2\mathsf{CF}_{k+1}(A)$.

When $p$ is even, then $\frac{p}{2} > \mathsf{CF}_{k+1}(A)$. By the definition of colored Frobenius numbers, $\frac{p}{2}$ has a $(k+1)$-chromatic solution in the colored numerical semigroup $\sg(A_{1},A_{2},\ldots,A_{\ell})$. Hence we can construct a $(k+1)$-chromatic solution of $p$ in the colored numerical semigroup $\sg(2A_{1},2A_{2},\ldots,2A_{\ell},\{b\})$.

When $p$ is odd, then $\frac{p-b}{2} > \mathsf{CF}_{k}(A)$. By the definition of colored Frobenius numbers, $\frac{p-b}{2}$ has a $k$-chromatic solution in the colored numerical semigroup $\sg(A_{1},A_{2},\ldots,A_{\ell})$. Hence we can construct a $k+1$-chromatic solution of $p$ in the colored numerical semigroup $\sg(2A_{1},2A_{2},\ldots,2A_{\ell},\{b\})$.

If $p = 2\mathsf{CF}_{k}(A) + b$, then $p-b = 2\mathsf{CF}_{k}(A)$. By the definition, $\frac{p-b}{2}$ has no $k$-chromatic solution in colored numerical semigroup $\sg(A_{1},A_{2},\ldots,A_{\ell})$. Hence, $p-b$ has no $k$-chromatic solution in the colored semigroup $\sg(2A_{1},2A_{2},\ldots,2A_{\ell})$. Since $b \geq  2\mathsf{CF}_{k}(A)$, $p - tb$ will be negative for $t \geq 2$. Overall, $p$ has no $k+1$-chromatic solution in the colored numerical semigroup $\sg(2A_{1},2A_{2},\ldots,2A_{\ell},\{b\})$.

We now consider claim~(b), we will prove the following statement:\ for any $b > \mathsf{CF}_{\ell}(A)$,
$$\mathsf{CF}_{\ell+1}(A_{1},\ldots,A_{\ell},\{b\}) = \mathsf{CF}_{\ell}(A_{1},\ldots,A_{\ell}) + b.$$

If a number $p >  \mathsf{CF}_{\ell}(A) + b$, then $p - b > \mathsf{CF}_{\ell}(A)$. By the definition of the colored Frobenius numbers, $p-b$ has a $\ell$-chromatic solution in $\sg(A_{1},\ldots,A_{\ell})$, hence $p$ has a $k+1$-chromatic solution in $\sg(A_{1},\ldots,A_{\ell},b)$.

If $p = \mathsf{CF}_{\ell}(A) + b$, then $p-b = \mathsf{CF}_{\ell}(A)$. By the definition, $p-b$ has no $k$-chromatic solution in $\sg(A_{1},\ldots,A_{\ell})$. Since $b \geq  \mathsf{CF}_{\ell}(A)$, $p - tb$ will be negative for $t \geq 2$. Therefore, $p$ has no $k+1$-chromatic solution in $\sg(A_{1},\ldots,A_{\ell},\{b\})$.

When $\ell > k$, we can choose $b \ge 2(\min A_1 + \cdots + \min A_\ell + F(A))$, and when $\ell = k$, we can choose $b \geq \min A_1 + \cdots + \min A_\ell + F(A).$
By Corollary~\ref{c:welldef} and the definition of colored Frobenius numbers, when $\ell > k$, 
$$b \geq 2(\min A_1 + \cdots + \min A_\ell + F(A)) \geq 2\mathsf{CF}_{\ell}(A) \geq 2\mathsf{CF}_{k+1}(A)\geq 2\mathsf{CF}_{k}(A);$$
when $\ell = k$, 
$$b \geq \min A_1 + \cdots + \min A_\ell + F(A) \geq \mathsf{CF}_{\ell}(A).$$
Clearly, these $b$'s satisfy the requirements.

To complete the proof, we note that since $F(A)$ has some trivial bounds like product of $a_{i}$'s and there are efficient algorithms to compute the minimum of a set, $b$ can be easily found in polynomial-time.
\end{proof}

\subsection{Counting chromatic solutions}
\label{subsec:counting}

In the remainder of this paper, we examine
$$f_{k}(b;A_{1},\ldots,A_{\ell})= \# \left\{ \text{k-chromatic solutions of } b \right\}.$$
for a given colored numerical semigroup $S = \sg(A_{1},\ldots,A_{\ell})$.  

Recall that a function $g:\ZZ_{\ge 0} \to \CC$ is said to be \emph{quasipolynomial} of period $N$ if
$$
g(n) = p_i(n)
\qquad \text{whenever} \qquad
n \equiv i \bmod N,
$$
for some polynomials $p_{0},\ldots,p_{N-1}$.  Moreover, a function $f:\ZZ_{\ge 0} \to \CC$ is \emph{eventually quasipolynomial} if there exists a quasipolynomial function $g$ such that $f(n) = g(n)$ for all but finitely many $n \in \ZZ_{\ge 0}$.

Fix a field $\kk$ and let $R = \kk[x_1, \ldots, x_n]$.  A $\ZZ_{\ge 0}$-grading of $R$ is specified by choosing $\deg(x_i) = a_i \in \ZZ_{\ge 0}$ and then defining 
$$\deg(x_{1}^{\xi_{1}}x_{2}^{\xi_{2}}\cdots x_{n}^{\xi_{n}}) = \xi_{1}a_{1} + \xi_{2}a_{2} + \cdots + \xi_{n}a_{n}.$$
An element of $R$ is \emph{homogeneous of degree $b$} if all of its terms have degree $b$, and an ideal $I \subseteq R$ is \emph{homogeneous} if $I$ can be generated (as an ideal) by homogeneous elements.  The \emph{$b$-graded piece} of a homogeneous ideal $I$ is 
$$I_b = \spann_\kk \{r \in I : r \text{ is homogeneous of degree } b \},$$
and the \emph{Hilbert function} of $I$ is the function $h_I:\ZZ_{\ge 0} \to \ZZ_{\ge 0}$ given by 
$h_I(b) = \dim_\kk I_b.$
For example, if $R = \kk[x, y]$, $\deg(x) = 2$, $\deg(y) = 3$, and $I = \<x^5, y^5\>$, then 
$$
R_{18} = \spann_\kk\{x^9, x^6 y^2, x^3 y^4, y^6\},
$$
so $h_R(18) = 4$ and $h_I(18) = 3$.  
We direct the reader to~\cite{millersturmfels} for background on Hilbert functions, and on the following theorem of Hilbert.  

\begin{thm}[Hilbert]\label{Theorem:Hilb}
Fix a $\ZZ_{\geq 0}$-graded polynomial ring $R$ over a field~$\kk$ and a homogeneous ideal $I \subseteq R$. The Hilbert function of $I$ is eventually quasipolynomial.
\end{thm}


\begin{proof}[Proof of Theorem~\ref{mt:numerical}\eqref{cardcolorsolutions}]
Fix a field $\kk$, let $R = \kk[x_1, x_2, \ldots, x_n]$, and fix a colored numerical semigroup $S = \sg(A_1, \ldots, A_\ell)$ with $A = \{a_1, \ldots, a_n\}$.  The map
\begin{align*}
\psi: \{ \text{monomials in }R \} &\longrightarrow S \\
x_{1}^{\xi_{1}}x_{2}^{\xi_{2}}\cdots x_{n}^{\xi_{n}} &\longmapsto \xi_{1}a_{1} + \xi_{2}a_{2} + \cdots + \xi_{n}a_{n}.
\end{align*}
induces a natural bijection between the monomials in $R$ and representations of elements of $S$.  The preimage of $\psi$ induces a grading on $R$ that sets $\deg(x_{i}) = a_{i}$ for each $i$, with one graded piece $R_b$ for each $b \in S$, and the monomials in $R_b$ each correspond to a representation of $b$.  


Now, a monomial $x_{1}^{\xi_{1}}x_{2}^{\xi_{2}}\cdots x_{n}^{\xi_{n}} \in R$ corresponds under $\psi$ to a $k$-chromatic representation precisely when the nonzero $\xi_i$'s lie in at least $k$ distinct color classes.  As such, if $x_{1}^{\xi_{1}}x_{2}^{\xi_{2}}\cdots x_{n}^{\xi_{n}}$ corresponds to a $k$-chromatic representation, then so does any monomial multiple (this is essentially the proof of Theorem~\ref{t:structure}).  
As such, the monomials in
$$I = \<x_{1}^{\xi_{1}}x_{2}^{\xi_{2}}\cdots x_{n}^{\xi_{n}} : \xi_{1}a_{1} + \xi_{2}a_{2} + \cdots + \xi_{n}a_{n} \text{ is $k$-chromatic}\>,$$
are precisely those that correspond to a $k$-chromatic representation under $\psi$, and thus the number of monomials in $I$ of degree $b$ is exactly $f_k(b; A_1, \ldots, A_\ell)$.  Applying Hilbert's theorem completes the proof.  
\end{proof}

\section*{Acknowledgements}
The authors would like to thank I.~Aliev, P.~Garc\'ia-S\'anchez, and J.~Gubeladze for insightful communications. We are also grateful to anonymous referees for their detailed comments and suggestions. The first and third authors were partially supported by NSF grant DMS-1818969.

\bibliographystyle{siam}
\bibliography{color}

\end{document}